\documentclass[12pt]{amsart}
\usepackage{amsmath,amsthm,amsfonts,amssymb,latexsym,enumerate}
\theoremstyle{plain}

\newtheorem{thm}{Theorem}[section]
\newtheorem{theorem}[thm]{Theorem}
\newtheorem{lemma}[thm]{Lemma}

\newtheorem{cor}[thm]{Corollary}

\numberwithin{equation}{section}

\newcommand{\Aut}{\operatorname{Aut}}

\newcommand{\SL}{\operatorname{SL}}

\newcommand{\N}{\mathbf{N}}

\newcommand{\PSL}{\operatorname{PSL}}
\newcommand{\PGL}{\operatorname{PGL}}
\newcommand{\PSU}{\operatorname{PSU}}

\def\cent#1#2{{\bf C}_{#1}(#2)}

\def\oh#1#2{{O}_{#1}(#2)}

\def\ker#1{{\rm ker}(#1)}

    \def \mod#1{\, {\rm mod} \, #1 \, }
\def\sbs{\subseteq}

\newcommand{\Out}{{\mathrm {Out}}}

\newcommand{\LC}{{\mathcal L}}

\renewcommand{\mod}{\bmod \,}

\newcommand{\gen}[1]{\langle #1 \rangle}

\newcommand{\aut}{\mathrm{Aut}}

\def \aut#1{{\rm Aut}(#1)}
\def\cent#1#2{{C}_{#1}(#2)}

\def\ker#1{{\rm ker}(#1)}

   \def \mod#1{\, {\rm mod} \, #1 \, }
\def\sbs{\subseteq}

\begin{document}

\title[Centers of Sylow Subgroups]
{Centers of Sylow Subgroups and Automorphisms}

\author{George Glauberman}
\address {Department of Mathematics, University of Chicago, 5734 S. University Ave, Chicago,
IL 60637, USA}
\email{gg@math.uchicago.edu}

\author{Robert Guralnick}
\address{Department of Mathematics, University of Southern California,
Los Angeles, CA 90089-2532, USA}
\email{guralnic@usc.edu}

\author{Justin Lynd}
\address{Department of Mathematics 
University of Louisiana at Lafayette 
Maxim Doucet Hall 
Lafayette, LA 70504, USA}
\email{lynd@louisiana.edu}
 
\author{Gabriel Navarro}
\address{Departament of Mathematics, Universitat de Val\`encia, 46100 Burjassot,
Val\`encia, Spain}
\email{gabriel.navarro@uv.es}



\thanks{The first author thanks the
Simons Foundation very much for its support by a Collaboration grant.  The second author 
gratefully acknowledges the support of the NSF grant DMS-1600056.
The third author  was partially supported by NSA Young Investigator Grant H98230-14-1-0312
and by an AMS-Simons travel grant.
The research of the fourth author is supported by the Prometeo/Generalitat 
Valenciana,
Proyecto MTM2016-76196-P   and FEDER funds.}

\begin{abstract}
Suppose that $p$ is an odd prime and $G$ is a finite group 
having no normal non-trivial $p'$-subgroup.   We
show that if $a$ is an automorphism of $G$ of $p$-power order centralizing a Sylow $p$-group of $G$,
then  $a$ is inner.   
\end{abstract}

\maketitle

\tableofcontents

\section{Introduction}

Let $p$ be a prime.  There has been quite a lot of interest in the problem of characterizing 
automorphisms of order a power of $p$ centralizing a Sylow $p$-subgroup of a finite group
$G$.   

In particular, Question 14.1 of  the Kourovka notebook
\cite{Ko} which was posed in 1999 asked whether if $G$ has no non-trivial normal odd order subgroups
(i.e. $O_{2'}(G)=1$),
then for any such automorphism $a$ with $p=2$, $a^2$ is inner.  This had already been answered in a 
paper of Glauberman \cite[Corollary 8]{Gl2} in 1968.   By taking $G=A_n$ with $n \ge 6$ with
$n \equiv 2, 3 \mod 4$ and $g \in S_n$ a transposition,  one sees that it is not always
true that $a$ is inner.  

If $p$ is odd, then under the assumptions that $O_p(G)=O_{p'}(G)=1$, 
Gross \cite{Gr} showed that any such automorphism is inner.  Gross used the classification of finite simple groups while Glauberman did not.
In this note, we show how to extend the result of Gross allowing the possibility of nontrivial $O_p(G)$ (and using the classification
of finite simple groups).
This was conjectured by Gross in \cite{Gr} and a partial result was obtained by Murai \cite{Mu}.    We complete the proof and show:

\begin{theorem} 
\label{thm:main}
 Let $p$ be an odd prime and $L$ a finite group with $O_{p'}(L)=1$.
Suppose that $a$ is an automorphism of $L$ whose order is a power of $p$
and $a$ centralizes a Sylow $p$-subgroup of $L$.   Then $a$  is an inner 
automorphism of $L$.
\end{theorem}

It is not hard to see that Theorem \ref{thm:main}
 can fail if $O_{p'}(L) \ne1$. 
 For example, take $L = G \times O_{p'}(L)$ and choose $a$
to be a noninner automorphism of $O_{p'}(L)$, viewed as an automorphism of 
$L$ acting trivially on $G$.
 A consequence of
 Theorem \ref{thm:main} is the following. Recall that $F^*(G)$ is the generalized Fitting subgroup of $G$.

\begin{cor} 
\label{cor:center}   Let  $p$ be an odd prime, and let
 $G$ be a finite group with $O_{p'}(G)=1$. Let $P$ be 
a Sylow $p$-subgroup of $G$.  Then $Z(P) \le F^*(G)$.
\end{cor} 

\smallskip

If $p=2$, the analogous corollary is that if $G$ is a finite group with no odd
normal subgroups, then $Z(P)/(Z(P)\cap F^*(G))$ is of exponent $2$ (and because
of Glauberman's result, this does not depend upon the classification of finite
simple groups).   If $G=S_n$ with $n \equiv 2 \mod 4$, then each transposition is
central in some Sylow $2$-subgroup $P$ and so certainly it is not the case that
$Z(P) \le F^*(G)$. 

We will give two proofs of Theorem \ref{thm:main}.   The first proof uses the
Thompson subgroup and  the $Z_p^*$ theorem for $p$ odd (the analog of
Glauberman's $Z^*$ theorem for $p=2$).   In contrast to Glauberman's $Z^*$
theorem, the $Z_p^*$ theorem relies on the classification of finite simple
groups. 

The second proof reduces to the 
almost simple case (where we use Gross' result \cite{Gr})  and to a more subtle result
about automorphisms of quasisimple groups.  See Theorem \ref{quasisimple}. 

We mention that there is a result directly analogous to Theorem \ref{thm:main}
that applies to a centric linking system $\LC$ associated to saturated fusion
system over a $p$-group $S$ with $p$ odd. Namely, any automorphism of $\LC$
which restricts to the identity on $S$, appropriately defined, is ``conjugation
by" an element of $Z(S) \leq \Aut_\LC(S)$. This follows by combining
\cite[Proposition~III.5.12]{AschbacherKessarOliver2011}, which connects
automorphisms of $\LC$ with limits of the center functor, with
\cite[Theorem~3.4]{Oliver2013} or \cite[Theorem~1.1]{GL16}, which show that the
center functor is acyclic at odd primes.  Likewise, it is shown in work to
appear of the first and third authors that there is an analogue of
\cite[Corollary~8]{Gl2} for centric linking systems at the prime $2$.  These
purely local results do not require an appeal to the Classification.

The paper is organized as follows.  In the next section we introduce some notation
and prove a few preliminary results.   In the following section, we discuss the
$Z_p^*$ theorem and indicate some connections with Gross' result and then give
our first proof.  In the next section we prove Theorem \ref{quasisimple} (which is stronger
than the main theorem for quasi-simple groups) and then show how to deduce
Theorem \ref{thm:main} from the results about simple and quasi-simple groups.   Finally 
we deduce the corollary.

We will need to use detailed properties of automorphism groups of simple groups.
Most of these results were first obtained by Steinberg.   We refer the reader to the
reference \cite[2.5]{GLS}. 
In particular, if $p \ge 3$, then the only quasi-simple groups with nontrivial  outer automorphisms of order a power
of $p$ are groups of Lie type.  Moreover,  if $p \ge 5$,  we are either in type A and there may be
diagonal automorphisms,  or else only field automorphisms are possible.   If $p=3$,  there
are more cases with diagonal automorphisms and triality needs to be dealt with as well.  

In the final section, we prove Theorem \ref{zpperm}   on permutation groups that is equivalent to the $Z_p^*$
theorem. 

We thank I. M. Isaacs for collaborating with us on an earlier version where we gave a new
proof of Glauberman's result and Gross' result.  
We also thank Richard Lyons, Gunter Malle and Ron Solomon for comments
on an earlier version of this paper.  

\section{Notation and Preliminary Results}

Let $G$ be a finite group and $p$ a prime.    We recall some notation (see \cite{I} for 
more details).   The maximal normal $p$-subgroup of $G$ is denoted
by  $O_{p}(G)$  and $O_{p'}(G)$ is the maximal normal $p'$-subgroup
of $G$.   The Fitting subgroup, $F(G)$, is the maximal normal nilpotent subgroup
of $G$ and is the direct product of the subgroups $O_r(G)$ where $r$ ranges over all prime
divisors of $|G|$.   

A quasi-simple group is a group $Q$ such that $Q$ is perfect (i.e. $Q=[Q,Q]$)  and $Q/Z(Q)$ is
a nonabelian simple group.   A component of $G$ is a subnormal quasi-simple
subgroup.  Then $E(G)$ is the subgroup of $G$ generated by all components of $G$
(and is the central product of all the components of $G$).  The generalized Fitting
subgroup $F^*(G)$ of $G$ is the central product $E(G)F(G)$.   It has the very important
property that $C_G(F^*(G)) = Z(F^*(G)) =Z(F(G))$.

\begin{lemma}  \label{gab-marty}  
Suppose that a finite group $L$ acts via automorphisms on
   a finite  abelian  $M$.
   Let $N$ be an $L$-invariant subgroup
   of $M$  such that
   $L$ acts trivially on the $p$-group  $M/N$. Assume that $M=NC_M(T)$
    for some Sylow $p$-subgroup
   $T$ of $L$. Then $M=NC_M(L)$. 
   \end{lemma}

\begin{proof}
Let $m \in C_M(T)$.
Thus,  $X:=\{m^{\ell}  \mid \ell \in L\}$ is an $L$-invariant subset of $M$ of cardinality 
$e$ with $e$ prime to $p$.
Let $m' = \prod_{x \in X} x$.  Clearly, 
$m' \in C_M(L)$. Using that $L$
acts trivially on $M/N$, we have that 
 $m' N = m^eN$. Then $m^e \in \cent ML N$. Since
 $e$ is prime to $p$, we have that
  the map $x \mapsto x^e$ is a bijection
 $C_M(T)/C_N(T)  \rightarrow C_M(T)/C_N(T)$. 
 Therefore $C_M(T) \sbs \cent ML N$, and we deduce that
 $M= \cent ML N$. \end{proof}

The following handles an easy case of the theorem and gives a reduction to the case that
$L=E(L)T$ where $T$ is a Sylow $p$-subgroup of $L$.   The next two results hold without
assuming that $a$ has order a power of $p$, but that is the important case and the only
case we use. 

\begin{lemma}\label{2.2}  Let $L$ be a finite group, $p$ prime and $O_{p'}(L)=1$.  If $a$ is
an automorphism of order a power of $p$ of $L$ centralizing $E(L)T$ with $T$ a Sylow $p$-subgroup of
$L$, then $a$ induces an inner automorphism of $L$.
\end{lemma}

\begin{proof}  Let $G =  L \langle a \rangle$ be the semidirect product. 
Notice that if $Q$ is a component of $G$, since $Q$ is perfect
and $Q/(Q \cap L)$ is cyclic, we have that $Q \sbs L$. Thus  
 the components of $G$ are the components of $L$.
Hence, we have that $F^*(G) = E(L)S$ where $S = F(G)$ is a $p$-group.

By hypothesis, $a$ is in the center of the Sylow $p$-subgroup  $T \langle a \rangle $
       of $G$ and $a$ centralizes $E(L)T$. Since $F(G) \leq T\langle a \rangle$, we have that  $a$
        centralizes $F^*(G) = E(G)F(G) =
     E(L)F(G)$.  So $a$  lies in $Z(F^*(G)) = Z(F(G)) = Z(S)$.

    Note that $M:=Z(S)$ is an abelian $p$-group
and $M/N$ is centralized by $L$ where $N = M \cap L$.   Now apply the previous lemma to  
conclude that there exists $z \in N\le L$ so that $az$ centralizes $L$ and so $a$ induces
conjugation by $z^{-1}$.
\end{proof}

A special case of the previous result is the following:

\begin{cor}  \label{2.3} Let $L$ be a finite group with $F^*(L)=O_p(L)$.  If $a$
is a nontrivial automorphism of $L$ of order a power of $p$ which centralizes
a Sylow $p$-subgroup $P$ of $L$, then $a$ is induced by conjugation
by an element of $Z(P)$.
\end{cor}

The following result will be used in studying quasi-simple groups.
If $g \in L$, then $g^L$ denotes the conjugacy class of $g$ in $L$.

\begin{lemma} \label{p-central}   Let $L$ be a finite group with  $Z$ a  central $p$-group.
Let $g \in L$ such that $|L:C_L(g)|$ is not divisible by $p$.   Then 
$g^L \cap gZ=\{g\}$.
\end{lemma}

\begin{proof}   
Suppose that $g^u = gz$  for some  $z \in Z$ and $u \in L$.   Thus $[g,u]=z$ and so $[g,u^r]=z^r$
for every integer $r$.  

Note also that $u$ normalizes the centralizer of $g$.
If $H=N_L(C_L(g))$, then the hypotheses
of the lemma are satisfied in $H$, and working by induction on $|L|$,
we may   assume that $H=L$.  Now, $L/C_L(g)$
is a $p'$-group. 
Let $r$ be the $p'$-part of the
order of $u$, and let $v = u^r$.  Then $v$ is in $C_L(g)$ and $g^v = gz^r$.   Hence, $z^r = 1$.  
Since $Z$ is
a $p$-group, $z = 1$. 
\end{proof}

\section{The $Z_p^*$ Theorem and a Proof}

Let $G$ be a group and $x \in H$ for some subgroup $H$ of $G$.  We say 
$x$ is isolated or weakly closed in $H$ with respect to $G$ if $x^G \cap H =\{x\}$.

\begin{thm} \label{zp*}  Let $G$ be a finite group with a Sylow $p$-subgroup $S$.
If $x \in S$ is  isolated with respect to $G$, then 
$x$ is central modulo $O_{p'}(G)$.
\end{thm}

 The result is usually
stated for elements of order $p$ but this implies the result for elements of order $p^a$ by
Lemma \ref{wc} below.  We also note that it is easy to see that if $x$ is isolated in
a Sylow $p$-subgroup, then $x$ is isolated in its centralizer (i.e. it does not commute
with any distinct conjugate). 

If $p=2$ and $x$ is an involution, then this is Glauberman's $Z^*$ theorem \cite{Glz} and does not
depend on the classification of finite simple groups.  So assume that $p$ is odd. 

It was observed by many that this follows from the classification of finite simple  groups (with some
effort) and always stated with the extra assumption that $x$ has order $p$.   It was proved
in \cite[Thm. 1]{Ar}, \cite[Thm. 4.1]{GR} and follows easily by \cite[7.8.2, 7.8.3]{GLS}.
Interestingly, the first two proofs both used \cite{Gr} to reduce from the almost simple
case (i.e. $F^*(G)$ is simple) to the simple case.  The last proof uses a different result about weakly closed subgroups
of order $p$.   The result for any $p$-element $x$ follows by the result for elements of order $p$ and
Lemma \ref{wc} below.

We give a quick sketch to indicate the connection with Gross' result.
There is no loss in assuming $O_{p'}(G)=1$ and then showing $x \in Z(G)$.   The result
is clear if $x \in O_p(G)$; so assume this is not the case.  
By induction we assume that $G$ is the   normal closure 
$\langle x^G \rangle$ of $x$  because otherwise we obtain $x \in Z(\langle  x^G  \rangle)
\le O_p(G)$.
If $E(G)=1$, then $F^*(G)=O_p(G)$ and $x \in C_G(F^*(G)) =Z(O_p(G))$ and the result holds.
Let  $Q$ be a component of  $G$.  Thus $|Q/Z(Q)|$ has order divisible by $p$, 
whence $x$ normalizes each
component and therefore so does $G$.  Note that $x$ must act nontrivially on $Q/Z(Q)$
(as $G$ is the normal closure of $x$).    
 
 Clearly $xZ(Q)$ is isolated in $G/Z(Q)$.   First assume that $Z(Q)=1$ and so
 we may assume that $G$ is almost simple. 
At this point, we can invoke \cite{Gr} to conclude that $x$ induces an inner automorphism
and so reduce to the simple case. 
In order to avoid that, we use \cite[7.8.2, 7.8.3]{GLS} (see also \cite[4.250]{Gor})
to reduce to the simple case and to a short list of possibilities.   In all these cases,
there is an involution inverting $x$ and the result follows.  

More generally this shows that $xZ(Q)$ is central in $G/Z(Q)$ whence $x \in O_p(G)$
which contradicts the fact that $G=\langle x^G \rangle$. 

One can also prove the result more directly after reducing to the case of simple groups.
If $G$ is alternating one sees that $N_G(\langle x \rangle ) \ne C_G(\langle x \rangle)$
for any nontrivial $p$-element $x$, whence $x$ is not isolated.  One checks the sporadic
groups directly.  If $G$ is a finite group of Lie type in characteristic $p$, the center consists
of root elements (or products of commuting short and long root elements in a few cases)
and it is easy to check.   If $G$ is a finite group of Lie type in characteristic $r \ne p$, then
in most groups, we have $-1$ in the Weyl group and every semisimple element of odd order is real,
whence not isolated.   This leaves the cases $\PSL_n(q), n > 2, \PSU_n(q), n > 2$ and 
orthogonal groups in dimension $2m$ with $m > 3$ odd and $E_6$ and ${^2}E_6(q)$.
The argument for the classical groups is an easy linear algebra argument and the group
of type $E_6$ follow by inspection of normalizers of maximal tori.  See \cite{GMN}
for similar arguments (proving a somewhat different result).

\begin{lemma}
\label{wc}
Let $G$ be a finite group with Sylow $p$-subgroup $S$. If $a \in S$ is weakly
closed in $S$ with respect to $G$, then any power of $a$ has the same property.
\end{lemma}

\begin{proof}
The hypothesis implies that $a \in Z(S)$. Let $b$ be a power of $a$, and let $g
\in G$ with $b_1 := b^{g^{-1}} \in S$.    Note that $a$ commutes with $b_1$ and
so $a^g$ commutes with $b$.   So $a^g$ is a $p$-element in $C_G(b)$ and
$S$ is a Sylow $p$-subgroup of $C_G(b)$.  Thus, $a^{gh} \in S$ for some $h \in C_G(b)$.
Since $a$ is weakly closed in $S$, $a^{gh} = a$ and so also $b^{gh}=b$.  Thus, $b^g=b$
whence $b_1=b$ and so $b$ is isolated in $S$.
\end{proof}

We can now prove the main theorem. 

\begin{theorem}
Let $p$ be an odd prime, and let $G$ be a finite group with Sylow $p$-subgroup
$S$. Let $a$ be an automorphism of $G$ of $p$-power order which centralizes
$S$. If $O_{p'}(G) = 1$, then $a$ is inner.
\end{theorem}
\begin{proof}
The argument is similar in part to the proof of
\cite[Lemma~8.2]{GL16}.  We induct on the order of $a$.  Let
$\hat{G} = G\gen{a}$ be the semidirect product, and let $\hat{S} = S\gen{a}$.
For each subgroup $X$ of $\hat{S}$, 
denote by $J(X)$ the Thompson subgroup of $X$ generated by the abelian subgroups of $X$ of
maximum order. Set $\hat{D} = Z(J(\hat{S}))$ and $D = \hat{D} \cap S$ for
short. Then $\hat{S}$ is Sylow in $\hat{G}$ and $\gen{a} \leq Z(\hat{S}) \leq
\hat{D}$.  So $\hat{X} = X \times \gen{a}$ for $X \in \{S,D\}$. 

Let $\hat{H} = N_{\hat{G}}(J(\hat{S}))$ in which $\hat{S}$ is Sylow, and let
$n$ be the index of $\hat{S}$ in $\hat{H}$. As $\hat{G}/G$ is abelian and
$\hat{D}$ is normal in $\hat{H}$,
\[
[\hat{D}, \hat{H}] \leq G \cap \hat{D} = D,
\]
so
\begin{eqnarray}
\label{HcentDD}
\text{$\hat{H}$ centralizes $\hat{D}/D$.}
\end{eqnarray}

Consider the norm/transfer/trace map
\[
\N = \N_{\hat{S}}^{\hat{H}} \colon C_{\hat{D}}(\hat{S}) \to C_{\hat{D}}(\hat{H})
\]
defined by setting
\[
\N(d) = \prod_{h \in [\hat{H}/\hat{S}]} d^{h}.
\]
By \eqref{HcentDD},
\[
\N(a) \equiv a^{n} \pmod{D}. 
\]
Since $|a|$ is coprime to $n$, the restriction $\N_{\gen{a}}$ is injective, and
we may choose $m \geq 1$ with $\N(a^m) \equiv a \pmod{D}$.  Thus we may find $z
\in D$ such that $az = \N(a^m) \in C_{\hat{D}}(\hat{H})$. Since $az \in
C_{\hat{D}}(\hat{H}) \leq Z(\hat{S})$ and $a \in Z(\hat{S})$, we see that $z
\in Z(\hat{S}) \cap G = Z(S)$.

Set $a_1 = az$; it has order $|a|$. By construction $a_1$ is weakly closed in
$\hat{S}$ with respect to $\hat{H}$. Since $p$ is odd, $a_1$ is weakly closed
in $\hat{S}$ with respect to $\hat{G}$ by \cite[Theorem~14.4]{Gla71}.

Let $b_1$ be a power of $a_1$ having order $p$.  Lemma~\ref{wc} gives that
$b_1$ is also weakly closed in $\hat{S}$ with respect to $\hat{G}$. So the
assumption $O_{p'}(G) = 1$ and the $Z^*_p$-theorem yield
$b_1 \in Z(\hat{G})$. This shows that conjugation by $a_1$ induces an
automorphism of $G$ of order at most $|a|/p$ centralizing $S$, and so
conjugation by $a_1$ is inner by induction. It follows that $a$ is inner.
\end{proof}

\section{Almost Simple Groups}

The following two theorems about simple and quasi-simple groups
	provide the key to give a second proof of our main results.   The first is
	a result of Gross \cite{Gr}, while the other is new and may be of independent
	interest.   Both results depend upon the classification of finite simple groups.

\begin{thm}  \label{simple}  
Let $p$ be an odd prime, and suppose $L$ is a finite
	nonabelian simple group with order divisible
	by $p$. Also, let $a$ be an automorphism of
	$L$ that has $p$-power order, and assume that
	$a$ centralizes a Sylow $p$-subgroup of $L$. Then
	$a$ is an inner automorphism of $L$.
\end{thm}

\begin{thm} \label{quasisimple}  Let $p$ be an odd prime and suppose
that $Q$ is a finite quasisimple group with center $Z$.    Let $P$ be a Sylow $p$-subgroup of $Q$
and let $x \in Z(P)$.   Let $H$ be the largest subgroup of $\aut{Q}$ with $[H,x] \le Z$ (i.e. 
$H=C_{\aut Q} (xZ/Z)$).   Then there exists $y \in xZ$ a $p$-element such that $H$ centralizes $y$.
\end{thm}

This gives the following corollary which is used in the second  proof.   The corollary
is an immediate consequence of the theorem by noting that an automorphism $\sigma$ of $Q$ 
commutes with conjugation by $x$ if and only $[\sigma, x ] \in Z$. 

\begin{cor} \label{quasi-simple}   Let $Q$ be a quasi-simple group with center $Z$ a $p$-group.
Suppose that $x \in Z(P)$ with $P$ a Sylow $p$-subgroup of $Q$.    There exists $y \in xZ$ such
that  if $\sigma$ is an automorphism
of $Q$ that commutes with conjugation by $x$, then
$\sigma$ fixes $y$. 
 \end{cor}

The remainder of this section is devoted to proving Theorem \ref{quasisimple}.    We
are assuming that Theorem \ref{simple} holds.

We first note:

\begin{lemma}  Let $Q$ be a quasisimple group.   If $p$ does not divide
$|Z(Q)|$ or $p$ does not divide $|\Out(Q)|$, then Theorem \ref{quasisimple} holds for $Q$.
\end{lemma}

\begin{proof}   Suppose that $Z=Z(Q)$ is a $p'$-group.   Then $\langle x \rangle$
is the Sylow $p$-subgroup of $\langle x, Z \rangle$ and so $[H,x] \le Z$ implies that
$[H,x]=1$.   Indeed, the same argument shows that by passing to $Q/O_{p'}(Z)$, we 
may assume that $Z$ is a nontrivial $p$-group.  

By Lemma \ref{gab-marty}, it suffices to prove the result for a Sylow $p$-subgroup 
$R$ of $H$.    If $p$ does not
divide $|\Out(Q)|$, then $R$ induces inner automorphisms on $Q$.   By Lemma \ref{p-central},
it follows that $R$ centralizes $x$.
\end{proof}

The previous  result shows that Theorem \ref{quasisimple} holds when $Q/Z$ is an alternating
or sporadic group (since the outer automorphism group has order $1,2$ or $4$ \cite[5.2.1, Table 5.3]{GLS}).
Thus we may assume that $L:=Q/Z$ is a finite group of Lie type and moreover that 
$Z$ is a nontrivial $p$-group. 

If the characteristic of $L$ is $p$, then almost always the Schur multiplier has order prime to $p$.
Using \cite{GLS}
it is straightforward to check in the few cases where $p$ does divide the order
of the Schur multiplier, $p$ does not divide the order of the
outer automorphism group, whence Theorem \ref{quasisimple} holds. 

Thus, we may assume that $L$ is a finite group of Lie type in characteristic $r \ne p$.
By \cite{GLS}, the only $L$ such that $p$ divides both the order of the Schur multiplier and the
outer automorphism group are:

\begin{enumerate}
\item $L=\PSL(d,q)$ and $p$ divides $(d, q-1)$; or
\item $L=\PSU(d,q)$ and $p$ divides $(d,q+1)$; or
\item  $p=3$ and $L=E_6(q)$ and $3$ divides $q-1$ or 
$L={^2}E_6(q)$ and $3$ divides $q+1$.
\end{enumerate}

Note that in all cases the Schur multiplier of $L$ is cyclic and so $Z$ is cyclic.   
In the last case above,  $Z(P)=Z$ by \cite{Lu} whence the theorem holds in that case.

We next prove an elementary result that is the key to proving Theorem \ref{quasisimple}.
The statement of the result is almost as long as the proof. 

\begin{lemma} \label{wreathlemma}  Let $c$ and $d$ be positive integers. 
Let $p$ be an odd prime and let $C$ be a cyclic group of order $p^c$.   Let $e$ be a positive integer
with $e \equiv 1 \pmod p$.   
Set $M=C^d$ (the direct sum of $d$ copies of $C$) and view $M$ as a module for 
$S_d \times \langle \sigma \rangle$ where $S_d$ acts on $M$ by permuting the coordinates of $M$
and $\sigma(x)=ex$ for all $x \in M$.   Let $\epsilon: M \rightarrow C$ be the augmentation map
(i.e. the sum of the coordinates) and $M_0 = \ker{\epsilon}$.   Let $Z$ be the group of fixed points of $S_d$
on $M$ and set $Z_0= Z \cap M_0$.   Let $Q$ be a Sylow $p$-subgroup of $S_d$.  
Let $M_1 = \{x \in M_0 | [x,Q] \le Z_0\}$.   Then 
$$
\{x \in M_1 | [x, \sigma] \in Z_0\} = Z_0 + C_{M_1}(\sigma).
$$
\end{lemma}

\begin{proof}  Let $q=p^b$ be the largest power of $p$ dividing $d$.
If $q=1$, then $M = M_0 \oplus Z$ and the result is clear.  So assume that $q > 1$.

Note that if $[x, \sigma] \in Z$, then $x = (x_1, \ldots, x_d)$ where $x_i=w + s_i$ with $w, s_i \in C$
and  $(e-1)s_i=0$.  

First suppose that $q \ne d$.  Thus, $Q$ has more than one orbit and so we see
that $M_1$ consists of those elements in $M_0$ in which the coordinates are constant
on each orbit of $Q$.  Thus, if $x \in M_1$ and $[\sigma, x] \in Z_0$, it follows that
$x = (x_1, \ldots, x_d)$ where $x_i=w + s_i$ and each $s_i$ occurs a multiple of $q$ times
(since the coordinate is constant on each orbit of $Q$.   Thus
$ dw = - qt$ where $t$ is in the subgroup generated by the $s_i$ and in particular $(e-1)t = 0$.
Suppose that $dw \ne 0$.   Then since $q$ is the largest power of $p$ dividing $d$, it follows that
$w$ and $t$ generate the same subgroup of $C$ and so $(e-1)w=0$ and so $w$ and $x$ are centralized
by $\sigma$ and the result holds.   If $dw=0$, then $x =(w, \ldots, w) + (s_1, \ldots, s_d) \in Z_0 + C_{M_1}(\sigma)$.

Finally suppose that $q=d$.   
Then $Q$ permutes $p$ blocks of imprimitivity (possibly $q=p$) of size $q/p$ (we take the blocks
to consist of consecutive integers).    Let $Q_1$ be the subgroup
of index $p$ fixing each block.   Let $x \in M$. 
Then $[x, Q_1] \le  Z$ implies that all coordinates of $x$ on each block
are constant.  So write $x=(x_1, \ldots, x_p)$ where $x_i$ is a constant vector corresponding to the $i$th block.   Let 
$\rho \in Q$ be of order $p$ permuting the blocks.   We assume
that $\rho$ takes $(x_1, \ldots, x_p)$ to $(x_2, \ldots, x_p, x_1)$.
  Then $[x, \rho] \in Z$ implies that
$x=(y, \ldots, y) + (0, u, 2u, \ldots, (p-1)u)$ with $pu=0$.   Let $W=\{x \in M |[x,Q] \le Z\}$.  So we have shown
that $W = Z \oplus Z'$ where $Z'$ is the subgroup of order $p$ generated by  $(0, u, 2u, \ldots, (p-1)u)$
with $u$ any element of order $p$.  Note that $|Z'|=p$ and moreover $Z' \le M_0$ (since $p$ is odd)
and is centralized by $\sigma$.    
Thus $ M_1 = W \cap M_0 = Z_0 \oplus Z' = Z_0 + C_{M_1}(\sigma)$ and the result follows. 
\end{proof}

We now give the proof of Theorem \ref{quasisimple} 
 in the  case that $L=\PSL(d,q)$  with $p$ dividing $(d,q-1)$ and $q=r^e$.
The proof is identical for $L=\PSU(d,q)$ (with $p$ dividing $(d,q+1)$ -- using the fact that
$p$ is odd).   The idea is to reduce to working in the normalizer of a maximal torus and
then the result essentially follows by Lemma \ref{wreathlemma}. 

Let $R =\SL(d,q)$ with center $Z_2$.    Then we can take $Q=R/Z_1$ for some 
subgroup $Z_1 \le Z_2$ such that  $Z_2/Z_1=Z$ is a nontrivial $p$-group.

We actually prove a bit more than we require by working in $R$ rather than
$R/Z_1$.     Let $T$ be the diagonal subgroup of $R$
and let $P$ be a Sylow $p$-subgroup of $R$ contained in the normalizer of $T$.   
Note that the normalizer of $T$ is just $TS_d$ and we can take $P \le TW_1$ with
$W_1$ a Sylow $p$-subgroup of $A_d$. 
Let $x \in P$  with $[P,x] \le Z_2$.   It is straightforward to see that $x \in T$.

Let $H$ be the subgroup of $\aut{R}$ such $[H,x] \le Z_2$
and let $W$ be a Sylow $p$-subgroup of $H$.     Note that we can take 
$W$ so that $W = S(W_1 \times \langle \sigma \rangle)$ where $S$ is the Sylow $p$-subgroup of $T$
in $\PGL_d(q)$ (i.e. the corresponding split torus in $\PGL$, which in particular centralizes $T$), 
$W_1$ is as above and $\sigma$ is a standard Frobenius automorphism (of $p$-power order).   

By Lemma \ref{wreathlemma},  we can find $y \in xZ_2$ so that $\sigma$ centralizes $y$ 
and so replacing $x$ by $y$ we may assume that $\sigma$ centralizes $x$.  Thus, 
$[P,x] = [W,x]$.  In particular if $xZ_1$ is central in  $P/Z_1$, then $W$ centralizes $xZ_1$, whence
by an averaging argument (Lemma \ref{gab-marty}), $H$ centralizes $xZ_1$ as required.

This completes the proof of Theorem \ref{quasisimple}.

\section{Second Proof of the Theorem and Proof of the Corollary}

 \begin{thm}
Suppose that
$p$ is an odd prime and $L$ is a finite group with $\oh{p'} L=1$.
Suppose that $a \in \aut L$ has order a power of $p$
and $a$ centralizes a Sylow $p$-subgroup of $L$.
Then $a$ acts as an inner automorphism of $L$.
\end{thm}

 \begin{proof}
    By Lemma \ref{2.2} we may assume that $L=E(L)S$.
  Thus, the $L$ orbits of a component are precisely the $S$-orbits.
  Let $Q$ be a component of $L$.  Let $t$ be the number of conjugates of $Q$ in $G$.
  Since $O_{p'}(G)=1$, $S \cap Q$ is not contained in $Z(Q)$ and since $a$ centralizes
  $S$, $a$ normalizes $Q$. 
  
   Then $a$ induces an inner automorphism on
  $Q/Z(Q)$ by Theorem \ref{simple} and so on $Q$ since it is perfect.   Thus, by Corollary \ref{quasi-simple}, 
  there exists $q \in Q$ such that $aq$ centralizes $Q$ and moreover $q$ centralizes
  any automorphism of $Q$ that centralizes $qZ(Q)/Z)Q)$ in $Q/Z(Q)$.   Since $S$ centralizes
  $a$, it follows that $N_S(Q)$ centralizes $qZ(Q)/Z(Q)$ whence $N_S(Q)$ centralizes $q$.
     Thus, the set   of $S$-conjugates of $q$ consists of 
  $t$ elements with one in each of the $t$ conjugates of $Q$.  In particular, these conjugates commute
  and their product $b$ is thus centralized by $S$.   Moreover, since $qa$ centralizes $Q$ and 
  $S$ centralizes $a$, it follows that $ab$ centralizes the central product of the conjugates of $Q$
  as well as $S$.   
  
  Repeating this for each orbit of components of $L$, we see that there is a  $p$-element
  $c \in E(L)$ such that $ac$ centralizes $SE(L) = L$, whence $a$ induces conjugation
  by $c$ on $L$ and the result follows.
  \end{proof}

Finally we deduce Corollary \ref{cor:center} from Theorem \ref{thm:main}.

So let $G$ be a finite group with $O_{p'}(G)=1$ and let $P$ be 
a Sylow $p$-subgroup of $G$, where $p$ is odd.  Let $x \in Z(P)$.   Then conjugation
by $x$ induces an automorphism of $F^*(G)$ and centralizes a
Sylow $p$-subgroup of $F^*(G)$.   Thus, $x$ is inner on $F^*(G)$.
Thus,  $x \in F^*(G)C_G(F^*(G))= F^*(G)$ and the result follows.

\section{A permutation version of the $Z_p^*$ theorem}

Finally we  prove a   permutation group result that is essentially equivalent to
the $Z_p^*$ theorem.  In particular, for $p=2$, the proof does not require
the classification of finite simple groups.  

\begin{thm}   \label{zpperm} 
Suppose $p$ is a prime, $G$ is a transitive subgroup of $S_n$, and $G$
possesses a   $p$-element $g$ that has a unique fixed point w and
is central in $G_w$, the stabilizer of $w$.
Then  $N:=O_{p'}(G)$ is transitive, $G_w = C_G(g)$
and $G=NC_G(g)$.  
\end{thm}

\begin{proof}  
Since $g$ has a unique fixed point $w$,   $N_G(\langle g \rangle)$ also fixes
$w$.  By assumption, $G_w \le C_G(g)$, whence $C_G(g)=G_w$.  Let $x \in G$, and
assume $g^x \in C_G(g) = G_w$. Then $g^x$ fixes the unique point $w^x$, but as
$g^x \in G_w$, it also fixes $w$. Thus, $x \in G_w = C_G(g)$ by uniqueness, so
that $g^x = g$. This shows that $g$ is isolated in $C_G(g)$. 
Since $p$ divides $n-1$, $G_w$ contains a Sylow $p$-subgroup $P$ of $G$.  Thus
$g$ is isolated in $P$. By the $Z_p^*$-theorem, it follows that $g$ is central
modulo $N$.  

In particular, $M:=\langle N, g \rangle$ is normal in $G$ and $\langle g
\rangle$ is a Sylow $p$-subgroup of $M$.   By the Frattini argument,
$G=MN_G(\langle g \rangle) =NC_G(g) = NG_w$ and the theorem follows.
\end{proof}

Let us note the previous theorem implies the $Z_p^*$-theorem.  By the usual
reductions (as described after Theorem \ref{zp*}), it suffices to prove this
when $G$ is almost simple and its socle has order divisible by $p$.  Suppose
that $g \in G$ is a nontrivial $p$-element and $g$ is isolated in a Sylow
$p$-subgroup $P$.  Then $g$ is also isolated in $C_G(g)$ (for if $g^a \in
C_G(g)$, then $g$ and $g^a$ are in a Sylow $p$-subgroup $P^b$ of $C_G(g)$ and
so $g^{ab^{-1}}$ are both in $P$, whence $g^a = g^b =g$). Let $G$ act on the
left cosets of $C_G(g)$.    Since $G$ is almost simple, the action is faithful.
Since $g$ is isolated in $C_G(g)$,  $g$ has a unique fixed point in this
action.   By the theorem,  this implies that $O_{p'}(G)$ is transitive and
trivial, whence $g$ is central.

\end{document}